\documentclass[reqno,12pt]{amsart}
\usepackage[utf8]{inputenc}
\usepackage[english]{babel}
\usepackage{amssymb,amsmath,amsthm,amsfonts,xcolor,enumerate,hyperref,comment,longtable,cleveref}

\usepackage{times}
\usepackage{cite}
\usepackage{pdflscape}
\usepackage[mathcal]{euscript}
\usepackage{tikz}
\usepackage{hyperref}
\usepackage{cancel}
\usepackage{stmaryrd}

\usetikzlibrary{arrows}

\newtheorem{theorem}{Theorem}[section]

\newtheorem{proposition}[theorem]{Proposition}

\newtheorem{definition}[theorem]{Definition}

\newtheorem{cons}[theorem]{Consequence}

 \DeclareMathOperator{\id}{id}

\newenvironment{Proof}[1][Proof.]{\begin{trivlist}
\item[\hskip \labelsep {\bfseries #1}]}{\flushright
$\Box$\end{trivlist}}

\usepackage{stmaryrd}
\usepackage{xcolor}

\begin{document}

	\title{Rota-type operators on null-filiform associative algebras}

	\author{I.A. Karimjanov}
\address{[I.A. Karimjanov] Department of Matem\'aticas, Institute of Matem\'aticas, University of Santiago de Compostela, 15782, Spain.}
\email{iqboli@gmail.com}

	\author{I.   Kaygorodov}
\address{[I.   Kaygorodov] CMCC, Universidade Federal do ABC. Santo Andr\'e, Brasil.}
\email{kaygorodov.ivan@gmail.com}

	\author{M. Ladra}
\address{[M. Ladra] Department of Matem\'aticas, Institute of Matem\'aticas, University of Santiago de Compostela, 15782, Spain.}
\email{manuel.ladra@usc.es}

\

\

\thanks{This work was partially supported by Agencia Estatal de Investigaci\'on (Spain), grant MTM2016-79661-P (European FEDER
support included, UE); RFBR 17-01-00258; FAPESP 17/15437-6.  }

\begin{abstract}
We give the description of homogeneous Rota-Baxter operators, Reynolds operators, Nijenhuis operators, Average operators and differential operator of weight 1 of null-filiform associative algebras of arbitrary dimension.

\end{abstract}
\maketitle

\section{Introduction}

Motivated by the important roles played by various linear operators in the study of mathematics through their actions on objects, Rota \cite{Rota} posed the problem of
finding all possible algebraic identities that can be satisfied by a linear operator on an algebra,
henceforth called Rota's Classification Problem.
Operator identities that were interested to Rota included
endomorphism operator, differential operator, average operator, inverse average operator, Rota-Baxter operator of weight $\lambda$ and  Reynolds operator.
After Rota posed his problem, more operators have appeared, such as
differential operator of weight $\lambda$, Nijenhuis operator, Leroux's TD operator
(for more information, see \cite{GG}).
The pivotal roles played by the endomorphisms (such as in Galois theory) and derivations (such as in calculus) are well-known. Their abstractions have led to the concepts of difference algebra and differential algebra respectively. The other operators also found applications in a broad range of pure and applied mathematics, including combinatorics, probability
and mathematical physics \cite{Rota,Guo}.

Here we give the definition of  operators which we considering in the present paper.

\begin{enumerate}

    \item[$\bullet$]
Homomorphism:
$\phi(x)\phi(y)=\phi(xy)$
    \item[$\bullet$]
Rota-Baxter operator of weight $\lambda$:
$P(x)P(y)=P(xP(y)+P(x)y+\lambda xy)$

    \item[$\bullet$]
Reynolds operator: $P(x)P(y)=P(xP(y)+P(x)y-P(x)P(y))$

    \item[$\bullet$]
Nijenhuis operator: $P(x)P(y)=P(xP(y)+P(x)y-P(xy))$

    \item[$\bullet$]
Average operator: $P(x)P(y)=P(xP(y))$

    \item[$\bullet$]
Differential operator of weight $\lambda$:
$d(xy)=d(x)y+xd(y)+\lambda d(x)d(y)$

\noindent In the particular case of that the weight is $0$ we have a derivation.

\end{enumerate}

The description of Rota-Baxter operators and other Rota-type operators is a very difficult problem.
At the moment there are descriptions of
all Rota-Baxter operators on $3$-dimensional simple  Lie algebra \cite{PBG,pkit},
on $4$-dimensional simple associative algebra \cite{TZS} and some other algebras \cite{maz,bai,bgp}.
The study of some particular cases of Rota-Baxter operators was initiated in \cite{Yu,ZGR} and \cite{GLB}.
In these papers, monomial and homogeneous operators were researched.
After that, the study of homogeneous operators of Rota-Baxter was continued in $3$-Lie algebras \cite{BZ1,BZ2}.

The organization of this paper is as follows. After this introduction, in Section~\ref{S:prel} we give some definitions
and necessary well-known results on null-filiform associative  algebras for the development of the paper.
In Section~\ref{S:main} we establish the principal results on the previous Rota type operators on null-filiform associative algebras.

Throughout this paper we will work with algebras over the field of complex numbers $\mathbb{C}$.

\section{Preliminaries}\label{S:prel}

\textbf{Gradings}. In the literature, group gradings have been intensively studied in the last years, motivated in part by their application in
physics, geometry and topology where they appear as the natural framework for an algebraic model \cite{cc09,e10,p89,bgr}.
In particular, in
the field of mathematical physics, they play an important role in the theory of strings, color supergravity, Walsh functions or
electroweak interactions \cite{12,17}. Certain advantages of endowing with a grading to an algebra can also be found in \cite{19,20}.
However, gradings by means of weaker structures than a group have been considered in the literature just in a slightly way.
In the present paper we wish to study some linear operators related with some gradings on algebras.

Let us give the definition of homogeneous operators related with a grading on an algebra.

\begin{definition}
Let $\mathbb G$ be an abelian group,
 $\mathcal A$ a $\mathbb G$-graded algebra
and $k$  an element of $\mathbb G$.
A homogeneous operator $\Psi$ with degree $k$ on the algebra $\mathcal A$
is a linear operator on $\mathcal A$ satisfying
\[      \Psi_k(\mathcal{A}_m) \subseteq\mathcal{A}_{m+k}.
\]
\end{definition}

\textbf{Null-filiform algebras}.
The study of null-filiform and filiform algebras has a very big history.
There are many results dedicated by
null-filiform and filiform associative algebras
\cite{karel,20,bgr}

For an algebra $\mathcal{A}$ of an arbitrary variety, we consider the series
\[
\mathcal{A}^1=\mathcal{A}, \qquad \ \mathcal{A}^{i+1}=\sum\limits_{k=1}^{i}\mathcal{A}^k \mathcal{A}^{i+1-k}, \qquad i\geq 1.
\]

\begin{definition}
An $n$-dimensional algebra $\mathcal{A}$ is called null-filiform if $\dim \mathcal{A}^i=(n+ 1)-i,\ 1\leq i\leq n+1$.
\end{definition}

We say that  an  algebra $\mathcal{A}$ is \emph{nilpotent} if $\mathcal{A}^{i}=0$ for some $i \in \mathbb{N}$. The smallest integer satisfying $\mathcal{A}^{i}=0$ is called the  \emph{index of nilpotency} of $\mathcal{A}$.

It is easy to see that an algebra has a maximum nilpotency index if and only if it is null-filiform. For a nilpotent algebra, the condition of null-filiformity is equivalent to the condition that the algebra is one-generated.
%

All null-filiform associative algebras were described in  \cite[Proposition 5.3]{karel}:

\begin{theorem} An arbitrary $n$-dimensional null-filiform associative algebra is isomorphic to the algebra:
\[\mathcal{A} : \quad e_i e_j= e_{i+j}, \quad 2\leq i+j\leq n,\]
where $\{e_1, e_2, \dots, e_n\}$ is a basis of the algebra $\mathcal{A}$ and the omitted products vanish.
\end{theorem}

%

There is a natural grading on the null-filiform associative algebra $\mathcal{A}$,
\[\mathcal{A}=\bigoplus\limits_i\mathcal{A}_i,\]
where $\mathcal{A}_i=\mathcal{A}^i/\mathcal{A}^{i+1}, \ 1 \leq i\leq n-1$.

Now we are ready to give the definition of homogeneous operator of
natural grading on the null-filiform algebras.

\begin{definition}
A homogeneous operator $P_k$ with degree $k$ on the null-filiform associative algebra  is a operator on $\mathcal{A}$ of the following form
\begin{equation}\label{a}
P_k(e_i)=
\begin{cases}
       \psi(i)e_{i+k},  & 1\leq i\leq n-k, \\
       \psi(i)e_{i+k-n},  & n-k+1\leq i\leq n,
    \end{cases}
\end{equation}
where $\psi$ is a $\mathbb{C}$-valued function defined on $\mathbb{N}$.
\end{definition}

\section{Main results}\label{S:main}

\subsection{Derivations}

\begin{theorem}
Let $D$ be a derivation  on the null-filiform associative algebra.
Then
\[ D(e_k)= k \sum\limits_{i=1}^{n-k+1} \alpha_{i}  e_{k-1+i}.\]
\end{theorem}

\begin{Proof}
Let $D(e_1)=\sum \alpha_i e_i$.
It is easy to see that
\begin{equation}\label{nomer0}
D(e_k)=D(e_1^k)= \sum e_1 \cdots D(e_1) \cdots e_1= k e_{k-1}D(e_1)= k \sum\limits_{i=1}^{n-k+1} \alpha_{i}  e_{k-1+i}.\end{equation}
From associativity and commutativity of our algebra, it is easy to see that  every linear mapping defined by \eqref{nomer0} is a derivation.
The theorem is proved.
\end{Proof}

\subsection{Homomorphisms}

\begin{theorem}
Let $\phi$ be a homomorphism  on the null-filiform associative algebra.
Then \[ \phi(e_k)= \sum\limits_{l=k}^{n} \left( \sum\limits_{i_1+ \dots +i_l=k} \alpha_{i_1} \cdots \alpha_{i_l} e_{l}\right).\]

\end{theorem}

\begin{Proof}
Let $\phi(e_1)=\sum \alpha_i e_i$.
It is easy to see that
\begin{equation}\label{nomer1}
\phi(e_k)=\left(\phi(e_1)\right)^k=\left(\sum \alpha_i e_i\right)^k=\sum\limits_{l=k}^{n} \left( \sum\limits_{i_1+ \dots +i_l=k} \alpha_{i_1} \cdots \alpha_{i_l} e_{l}\right).\end{equation}
From associativity and commutativity of our algebra, it is easy to see that  every linear mapping defined by \eqref{nomer1} is a homomorphism.
The theorem is proved.

\end{Proof}

\subsection{Differential operator of weight $\mathbf{\lambda \neq 0} $} If $\lambda\neq0$ then without loss of generality we can assume that $\lambda=1$.
Using ideas from \cite[Lemma 2.2]{lambda}, we note that  $d$ is a differential operator of weight $1$ of the null-filiform associative algebra
if and only if $d+ \id$ is a homomorphism.
Now, using the description of all homomorphisms of the null-filiform associative algebra, we have the description of all differential operators of weight $1$.

\begin{theorem}
Let $d$ be a differential operator of weight $1$ on the null-filiform associative algebra.
Then \[ d(e_k)= \sum\limits_{l=k}^{n} \left( \sum\limits_{i_1+ \dots +i_l=k} \alpha_{i_1} \cdots \alpha_{i_l} e_{l}\right)-e_k.\]

\end{theorem}

\subsection{Homogeneous Rota-Baxter operator of weight $\mathbf{0}$}

\begin{theorem}  Let $P_0$ be a Rota-Baxter operator of weight $0$ with degree $0$ on the null-filiform associative algebra. Then
\begin{enumerate}
\item[(a)] If  $\psi(t)\neq0$ for some $1\leq t\leq \lfloor n/2\rfloor$ then
\[P_0(e_i)=
\begin{cases}
\frac{t\psi(t)}{i}e_i, & i=0 \ (\bmod \ t), \\
0, & \text{otherwise};
\end{cases}
\]
\item[(b)] If $\psi(t)=0$ for all $1\leq t\leq \lfloor n/2\rfloor$  then
\[P_0(e_i)=
\begin{cases}
\psi(r)e_r, & i=r, \\
0, & \text{otherwise},
\end{cases}
\]
for some fixed $\lfloor n/2\rfloor < r \leq n$.
\end{enumerate}

\end{theorem}

\begin{Proof}
Consider
\[P(e_i)P(e_j)=\psi(i)\psi(j)e_{i}e_{j}=\psi(i)\psi(j)e_{i+j}, \quad 2\leq i+j\leq n.\]
On the other hand
\[P(e_i)P(e_j)=P(e_iP(e_j)+P(e_i)e_j)=(\psi(i)+\psi(j))P(e_{i+j})=(\psi(i)+\psi(j))\psi(i+j)e_{i+j}.\]
Comparing the coefficients of the basic elements we derive
\begin{equation}\label{c}
    \psi(i)\psi(j)=(\psi(i)+\psi(j))\psi(i+j), \quad 2\leq i+j\leq n.
\end{equation}

\begin{enumerate}

\item[(a)] If $\psi(1)\neq0$ then we can prove the following equalities
by an induction on $i$:
\begin{equation}\label{d}
     \psi(i)=\frac{\psi(1)}{i}.
\end{equation}

Obviously,  Eq.~\eqref{d} holds for $i = 2$. Let us assume that the equality holds for $2 < i < n$, and
we will prove it for $i + 1$: using Eq.~\eqref{c} when $j=1$ we have
\[\psi(i+1)=\frac{\psi(1)\psi(i)}{\psi(1)+\psi(i)}=\frac{\frac{\psi(1)^2}{i}}{\frac{\psi(1)(i+1)}{i}}
=\frac{\psi(1)}{i+1}.\]
so the induction proves the Eq.~\eqref{d} for any
$i, \ 2\leq i\leq n$.

Let us suppose that $\psi(1)=0$ and
$\psi(t)\neq0$ for some $2\leq t\leq \lfloor n/2\rfloor$. With a similar
induction as the given for Eq.~\eqref{d}, it is easy to check that
the following equalities hold:
\[\psi(it)=\frac{\psi(t)}{i}, \qquad 2\leq i\leq \lfloor n/t\rfloor .\]

If $\{s/t\}\neq0$ we always can suppose $s=it+q$, where
$t+1\leq s\leq n,\ 1\leq q\leq t-1$. It follows from Eq.~\eqref{c} we
have
\[\psi(s)=\psi(it+q)=\frac{\psi(q)\psi(it)}{\psi(q)+\psi(it)}=0.\]

So we obtain \[\psi(s)=0, \quad t+1\leq s\leq n, \quad  \{s/t\}\neq0.\]

\item[(b)] Let $\psi(r)=0, \ 1\leq r\leq t-1$ and
$\psi(t)\neq0, \ \lfloor n/2\rfloor<t\leq n$. As previous case we can
always suppose $s=t+q$ where $1\leq q\leq n-t$ then
\[\psi(s)=\psi(t+q)=\frac{\psi(t)\psi(q)}{\psi(t)+\psi(q)}=0.\]
 Hence
we have $\psi(s)=0$ where $t+1\leq s\leq n$.
\end{enumerate}
\end{Proof}

\begin{proposition}  Let $P_k$ be a homogeneous Rota-Baxter operator of weight $0$ with degree $k \ (k\neq0)$ on the  null-filiform associative algebra.
Then
\[P_k(e_i)=
\begin{cases}
\psi(i)e_{i+k}, & 1\leq i \leq n-k, \\
0, & n-k+1\leq i \leq n,
\end{cases}
\]
and the function $\psi$ satisfies that:
\begin{equation}\label{b}
\psi(i)\psi(j)=(\psi(i)+\psi(j))\psi(i+j+k), \quad  2\leq i+j\leq n-2k.
\end{equation}
 \end{proposition}

\begin{Proof}
Let us consider
\[P(e_{n-k})P(e_{n-k+j})=\psi(n-k)\psi(n-k+j)e_ne_j=0, \quad 1\leq j\leq k.\]

By the other side
\begin{align*}
 P(e_{n-k})P(e_{n-k+j})&=P(e_{n-k}P(e_{n-k+j})+P(e_{n-k})e_{n-k+j})\\
&  = P(\psi(n-k+j)e_{n-k}e_j+\psi(n-k)e_ne_{n-k+j})\\
&=\psi(n-k+j)P(e_{n-k+j})=\psi^2(n-k+j)e_j.
\end{align*}
Hence $\psi(i)=0$ where $n-k+1\leq i\leq n$.

Analogously from
\[P(e_i)P(e_j)=\psi(i)\psi(j)e_{i+k}e_{j+k}=\psi(i)\psi(j)e_{i+j+2k}, \quad 2\leq i+j\leq n-2k,\]
and
\[P(e_i)P(e_j)=P(e_iP(e_j)+P(e_i)e_j)=(\psi(i)+\psi(j))P(e_{i+j+k})=(\psi(i)+\psi(j))\psi(i+j+k)e_{i+j+2k}\]
we obtain system of Eq.~\eqref{b}.
\end{Proof}

\begin{cons}  Let $P_k$ be a homogeneous Rota-Baxter operator of weight $0$ with degree $k \ (k\geq\lfloor n/2\rfloor)$ on the null-filiform associative algebra.
Then
\[P_k(e_i)=
\begin{cases}
\psi(i)e_{i+k}, & 1\leq i \leq n-k, \\
0, & n-k+1\leq i \leq n,
\end{cases}
\]
where $\psi(i)\in\mathbb{C}$.
 \end{cons}

%
%
%
%

Now we will consider an operator Rota-Baxter of weight $0$ with degree $1$.

Let us suppose that $\psi(1)\neq0$. Then by induction from Eq.~\eqref{b} where $j=1$ we have
$\psi(2i-1)=\frac{\psi(1)}{i}$.

Let $\psi(2)\neq0$. Then we obtain that
$\psi(2i)=\frac{2\psi(1)}{2i+1}$. Otherwise $\psi(2i)=0$ where $1\leq i\leq \lfloor n/2 \rfloor$.

Let us suppose that $\psi(1)=0$ and
$\psi(t)\neq0$ for some $2\leq t\leq \lfloor (n-2)/2\rfloor$. With a similar
induction as the given for previous cases, it is easy to check that
the following equalities hold:
\[\psi((t+1)s-1)=\frac{\psi(t)}{s}, \quad 1\leq s\leq \lfloor n/(t+1)\rfloor .\]

Moreover, \[\psi(m)=0, \quad t+1\leq m\leq n-1, \quad  \{(m+1)/(t+1)\}\neq0.\]

In case  $\psi(t)=0$ for all $1\leq t\leq \lfloor (n-2)/2\rfloor$, we have
that $\psi(r), \psi(r+1)$ are arbitrary elements for some fixed  $\lfloor (n-2)/2\rfloor < r \leq n-2$.

Summarizing we have the following theorem.

\begin{theorem}\label{d1}  Let $P_1$ be a Rota-Baxter operator of weight $0$ with degree $1$ on the null-filiform associative algebra. Then
\begin{enumerate}
\item[(a)] If $\psi(1)\psi(2)\neq0$  then $P_1(e_i)=\frac{2\psi(1)}{i+1}e_{i+1}$;
\item[(b)] If $\psi(1)\neq0$ and $\psi(2)=0$ then
\[P_1(e_i)=
\begin{cases}
\frac{2\psi(1)}{i+1}e_{i+1}, & i  \ \text{odd}, \\
0, & \text{otherwise};
\end{cases}
\]
\item[(c)] If $\psi(1)=0$  and $\psi(t)\neq0$ for some $2\leq t\leq \lfloor (n-2)/2\rfloor$ then
\[P_1(e_i)=
\begin{cases}
\frac{(t+1)\psi(t)}{i+1}e_{i+1}, & i+1=0 \ (\bmod \ t+1), \\
0, & \text{otherwise};
\end{cases}
\]
\item[(d)] If $\psi(t)=0$ for all $1\leq t\leq \lfloor (n-2)/2\rfloor$  then
\[P_1(e_i)=
\begin{cases}
\psi(r)e_{r+1}, & i=r, \\
\psi(r+1)e_{r+2}, & i=r+1, \\
0, & \text{otherwise},
\end{cases}
\]
for some fixed $\lfloor (n-2)/2\rfloor < r \leq n-2$.\end{enumerate}

\end{theorem}

\subsection{ Rota-Baxter operator of weight $ \mathbf{\lambda \neq 0}$}  If $\lambda\neq0$ then without loss of generality we can assume that $\lambda=1$.

\begin{theorem} Let $P_k$ be a Rota-Baxter operator of weight $1$ with degree $k$ on the null-filiform associative algebra.
Then
\begin{enumerate}
\item[(a)] If $k=0$ then $P_0(e_i)=\frac{\psi(1)^i}{(\psi(1)+1)^i-\psi(1)^i}e_i$;
\item[(b)] If $1\leq k < \lfloor n/2 \rfloor$ then $P_k(e_i)=0$;
\item[(c)] If $k\geq \lfloor n/2 \rfloor$ then
\[P_k(e_i)=
\begin{cases}
\psi(1) e_{1+k}, & i=1, \\
0, & \text{otherwise},
\end{cases}
\]
\end{enumerate}
where $1\leq i\leq n$ and $\psi(1)\in \mathbb{C}$.
\end{theorem}

\begin{Proof}

\begin{enumerate}

\item[(a)]  When $k=0$ the comparison of both linear combinations of the identity Rota-Baxter operator implies that:
\begin{equation}\label{e}
\psi(i)\psi(j)=(\psi(i)+\psi(j)+1)\psi(i+j), \quad 2\leq i+j\leq n.
\end{equation}

Let us denote that $a=\psi(1)$. We will show that
 \begin{equation}\label{f}
  \psi(s)=\frac{a^s}{(a+1)^s-a^s}, \quad 1\leq s\leq n.
  \end{equation}
Let us suppose that Eq.~\eqref{f} is true for $s$ and we consider for
$s+1$: using Eq.~\eqref{e} when $i=s$ and $j=1$
\begin{align*}
\psi(s+1)&=\frac{\psi(1)\psi(s)}{\psi(1)+\psi(s)+1}=\frac{a^{s+1}}{(a+1)^sa^s
}\div(a+\frac{a^s}{(a+1)^s-a^s}+1)\\
&= \frac{a^{s+1}}{(a+1)^s-a^s}\div\frac{(a+1)^sa-a^{s+1}+a^s+(a+1)^s-a^s}{(a+1)^s-a^s}\\
&= \frac{a^{s+1}}{(a+1)^{s+1}-a^{s+1}}.
\end{align*}
Finally we will check that Eq.~\eqref{f} satisfy to Eq.~\eqref{e}.

Considering
\[\psi(i)\psi(j)=\frac{a^{i+j}}{((a+1)^i-a^i)((a+1)^j-a^j)}.\]
From the other side
\begin{align*}
&(\psi(i)+\psi(j)+1)\psi(i+j)=(\frac{a^i}{(a+1)^i-a^i}+\frac{a^j}{(a+1)^j-a^j}+1)
\frac{a^{i+j}}{(a+1)^{i+j}-a^{i+j}}\\
&= \frac{(a^i((a+1)^j-a^j)+a^j((a+1)^i-a^i)+((a+1)^i-a^i)((a+1)^j-a^j))a^{i+j}}{((a+1)^i-a^i)((a+1)^j-a^j)((a+1)^{i+j}-a^{i+j})}\\
&=\frac{(a^i(a+1)^j+a^j(a+1)^i-2a^{i+j}+((a+1)^{i+j}-a^i(a+1)^j-a^j(a+1)^i)+a^{i+j})a^{i+j}}{((a+1)^i-a^i)((a+1)^j-a^j)((a+1)^{i+j}-a^{i+j})}\\
&=\frac{((a+1)^{i+j}-a^{i+j})a^{i+j}}{((a+1)^i-a^i)((a+1)^j-a^j)((a+1)^{i+j}-a^{i+j})}=\frac{a^{i+j}}{((a+1)^i-a^i)((a+1)^j-a^j)}.
\end{align*}

\item[](b)--(c) Let us consider
\[P(e_{n-k})P(e_{n-k+j})=\psi(n-k)\psi(n-k+j)e_ne_j=0, \quad 1\leq j\leq k.\]

By the other side
\begin{align*}
& P(e_{n-k})P(e_{n-k+j})=P(e_{n-k}P(e_{n-k+j})+P(e_{n-k})e_{n-k+j}+ e_{n-k}e_{n-k+j})\\
&=P\Big(\psi(n-k+j)e_{n-k}e_j+\psi(n-k)e_ne_{n-k+j}+
\begin{cases}
 e_{2n-2k+j}, & 2k-j\geq n, \\
0, & \text{otherwise},
\end{cases} \Big)
\\
&= P\Big(\psi(n-k+j)e_{n-k+j}+
\begin{cases}
 e_{2n-2k+j}, & 2k-j\geq n, \\
0, & \text{otherwise},
\end{cases}\Big)
\\
&=\psi^2(n-k+j)e_j+
\begin{cases}
\ \psi(2n-2k+j)e_{n-k+j}, & 2k-j\geq n, \\
0, & \text{otherwise}.
\end{cases}
\end{align*}
Hence $\psi(i)=0$ where $n-k+1\leq i\leq n$.

By the identity of the  operator Rota-Baxter  the function $\psi$ satisfies
\begin{align*}
\psi(i)\psi(j) &=(\psi(i)+\psi(j))\psi(i+j+k), && 2\leq i+j\leq n-2k,\\
\psi(i+j)& =0, && 2\leq i+j\leq n-k.
\end{align*}
Moreover, we can obtain $\psi(1)=0$ when $k<\lfloor n/2 \rfloor$.

\end{enumerate}
\end{Proof}
\subsection{Reynolds operator}

\begin{theorem}  Let $P_0$ be a homogeneous Reynolds operator with degree $0$ on the null-filiform associative algebra. Then

\begin{enumerate}
\item[(a)] If $\psi(t)\neq0$ for some $1\leq t\leq \lfloor n/2\rfloor$ then
\[P_0(e_i)=
\begin{cases}
\frac{t\psi(t)}{i-(i-t)\psi(t)}e_i, & i=0 \ (\bmod \ t), \\
0, & \text{otherwise};
\end{cases}
\]
\item[(b)] If $\psi(t)=0$ for all $1\leq t\leq \lfloor n/2\rfloor$  then
\[P_0(e_i)=
\begin{cases}
\psi(r)e_r, & i=r, \\
0, & \text{otherwise},
\end{cases}
\]
for some fixed $\lfloor n/2\rfloor < r \leq n$.
\end{enumerate}

\end{theorem}

\begin{proof}
By  the identity of the Reynolds operator the function $\psi$ satisfies
\begin{equation}\label{j}
  \psi(i)\psi(j)=(\psi(i)+\psi(j)-\psi(i)\psi(j))\psi(i+j), \quad 2\leq i+j\leq n.
\end{equation}
\item[(a)] Let $\psi(1)\neq0$. Then we can prove the following equalities
by an induction on $i$:
\begin{equation}\label{h}
 \psi_(i)=\frac{\psi(1)}{i-(i-1)\psi(1)}, \quad 2\leq i\leq n.
\end{equation}

     Obviously, the Eq.~\eqref{h} holds for $i = 2$. Let us assume that
the equality holds for $2 < i < n$, and we will prove it for $i + 1$:
using Eq.~\eqref{j} where $j=1$ we have
\begin{align*}
\psi(i+1)&=\frac{\psi(1)\psi(i)}{\psi(1)+\psi(i)-\psi(1)\psi(i)}=\frac{\psi^2(1)}{i-(i-1)\psi(1)}\div
\Big(\psi(1)+\frac{\psi(1)}{i-(i-1)\psi(1)}-\frac{\psi^2(1)}{i-(i-1)\psi(1)}\Big)\\
&=
\frac{\psi^2(1)}{i-(i-1)\psi(1)} \cdot \frac{i-(i-1)\psi(1)}{i\psi(1)-(i-1)\psi^2(1)+\psi(1)-\psi^2(1)}=
\frac{\psi(1)}{i+1-i\psi(1)}.
\end{align*}
 so the induction proves the equalities
(6) for any $i, \quad 1\leq i\leq n$.

Now we will check that Eq.~\eqref{h} satisfies Eq.~\eqref{j}

Considering
\[\psi(i)\psi(j)=\frac{\psi^2(1)}{(i-(i-1)\psi(1))(j-(j-1)\psi(1))}.\]

By the other side
\begin{align*}
&(\psi(i)+\psi(j)-\psi(i)\psi(j))\psi(i+j)=\Big(\frac{\psi(1)}{i-(i-1)\psi(1)}+
\frac{\psi(1)}{j-(j-1)\psi(1)}\\
&  \ \ -
\frac{\psi^2(1)}{(i-(i-1)\psi(1))(j-(j-1)\psi(1))}\Big)\cdot \frac{\psi(1)}{(i+j-(i+j-1)\psi(1))}\\
&=
\frac{\psi(1)(i-(i-1)\psi(1)+j-(j-1)\psi(1))-\psi^2(1)}{(i-(i-1)\psi(1))(j-(j-1)\psi(1))}\cdot \frac{\psi(1)}{(i+j-(i+j-1)\psi(1))}\\
&=\frac{\psi^2(1)}{(i-(i-1)\psi(1))(j-(j-1)\psi(1))}.
\end{align*}

 Let
us suppose that $\psi(1)=0$ and $\psi(t)\neq0$ for
some $2\leq t\leq \lfloor n/2\rfloor$. With a similar induction as the
given for Eq.~\eqref{h}, it is easy to check that the following
equalities hold:
\[\psi(it)=\frac{\psi(t)}{i-(i-1)\psi(t)}, \quad 2\leq i\leq [n/t].\]

If $\{s/t\}\neq0$ we always can write $s=it+q$, where
$t+1\leq s\leq n,\ 1\leq q\leq t-1$. From Eq.~\eqref{j} we
have
\[\psi(s)=\psi(it+q)=\frac{\psi(q)\psi(it)}{\psi(q)+\psi(it)-\psi(q)\psi(it)}=0.\]

So we have \[\psi(s)=0, \quad t+1\leq s\leq n, \quad  \{s/t\}\neq0.\]

\item[(b)] Let $\psi(r)=0, \ 1\leq r\leq t-1$ and
$\psi(t)\neq0, \ \lfloor n/2\rfloor<t\leq n$. As in the previous case we can
always write $s=t+q$ where $1\leq q\leq n-t$ then
\[\psi(s)=\psi(t+q)=\frac{\psi(t)\psi(q)}{\psi(t)+\psi(q)-\psi(t)\psi(q)}=0.\]
Hence we have $\psi(s)=0$ where $t+1\leq s\leq n$.
\end{proof}

\begin{proposition}  Let $P_k$ be a homogeneous Reynolds operator with degree $k \ (k\neq0)$ on the  null-filiform associative algebra.
Then
\[P_k(e_i)=
\begin{cases}
\psi(i)e_{i+k}, & 1\leq i \leq n-k, \\
0, & n-k+1\leq i \leq n,
\end{cases}
\]
and the function $\psi$ satisfies that
\begin{equation}\label{g}
\begin{aligned}
 &\psi(i)\psi(j)=(\psi(i)+\psi(j))\psi(i+j+k), && 2\leq i+j\leq n-2k, \\
 &\psi(i)\psi(j)\psi(i+j+2k) =0, && 2\leq i+j\leq n-3k.
 \end{aligned}
 \end{equation}
 \end{proposition}

\begin{Proof}
Let us consider for $1\leq j\leq k$
\[P(e_{n-k})P(e_{n-k+j})=\psi(n-k)\psi(n-k+j)e_ne_j=0. \]

By the other side
\begin{align*}
P(e_{n-k})P(e_{n-k+j})& =P(e_{n-k}P(e_{n-k+j})+P(e_{n-k})e_{n-k+j}-P(e_{n-k})P(e_{n-k+j}))\\
& = P(\psi(n-k+j)e_{n-k}e_j+\psi(n-k)e_ne_{n-k+j})\\
&=P(\psi(n-k+j)e_{n-k+j})=\psi^2(n-k+j)e_j.
\end{align*}
Hence $\psi(m)=0$ where $n-k+1\leq m\leq n$.

By the identity of the Reynolds operator  it is not difficult to obtain Eq.~\eqref{g}.
\end{Proof}

\begin{cons}Let $P_k$ be a homogeneous Reynolds operator with degree $k \ (k\geq\lfloor n/2\rfloor)$ on the null-filiform associative algebra.
Then
\[P_k(e_i)=
\begin{cases}
\psi(i)e_{i+k}, & 1\leq i \leq n-k, \\
0, & n-k+1\leq i \leq n,
\end{cases}
\]
where $\psi(i)\in\mathbb{C}$.
\end{cons}

Now we give next theorem.

\begin{theorem}  Let $P_1$ be a Reynolds operator  with degree $1$ on the null-filiform associative algebra. Then
\begin{enumerate}
\item[(a)] If $\psi(t)\neq0$ for some $1\leq t\leq \lfloor (n-2)/2\rfloor$ then
\[P_1(e_i)=
\begin{cases}
\frac{(t+1)\psi(t)}{i+1}e_{i+1}, & i+1=0 \ (\bmod \ t+1), \\
0, & \text{otherwise};
\end{cases}
\]
\item[(b)] If $\psi(t)=0$ for all $1\leq t\leq \lfloor (n-2)/2\rfloor$ then
\[P_1(e_i)=
\begin{cases}
\psi(r)e_{r+1}, & i=r, \\
\psi(r+1)e_{r+2}, & i=r+1, \\
0, & \text{otherwise},
\end{cases}
\]
for some fixed $\lfloor (n-2)/2\rfloor < r \leq n-2$.
\end{enumerate}

\end{theorem}
\begin{proof}The proof is carrying out by applying similar arguments as in the proof of  Theorem~\ref{d1}.
\end{proof}
\subsection{Nijenhuis operator}

\begin{theorem} Let $P_0$ be a homogeneous Nijenhuis operator with degree $0$ on the null-filiform associative algebra. Then
\[P_0(e_i)=a e_i, \quad 1\leq i\leq n,\]
where $a\in\mathbb{C}$.
\end{theorem}
\begin{Proof}
Considering \[P(e_i)P(e_j)=\psi(i)\psi(j)e_{i+j}, \quad 2\leq i+j\leq n.\]
On the other hand
\begin{align*}
P(e_i)P(e_j)&=P(e_iP(e_j)+P(e_i)e_j-P(e_ie_j))\\
&=P(\psi(j)e_{i+j}+\psi(i)e_{i+j}-\psi(i+j)e_{i+j})=(\psi(i)+\psi(j)-\psi(i+j))\psi(i+j)e_{i+j}.
\end{align*}

The comparison of both linear combinations implies that:
\[\psi(i)\psi(j)=(\psi(i)+\psi(j)-\psi(i+j))\psi(i+j), \quad 2\leq i+j\leq n.\]

From previous system equation when $i=j=1$
\[\psi^2(1)=2\psi(1)\psi(2)-\psi^2(2),\]
we obtain $\psi(2)=\psi(1)$.

Using previous system equations and by induction supposition we can prove that
\[\psi(i)=\psi(1), \quad 2\leq i\leq n.\]
Finally we denote by $a=\psi(1)$.
\end{Proof}

\begin{proposition}Let $P_k$ be a homogeneous Nijenhuis operator with degree $k \ (k\geq\lfloor n/2\rfloor)$ on  the null-filiform associative algebra.
Then
\[P_k(e_i)=
\begin{cases}
\psi(i)e_{i+k}, & 1\leq i \leq n-k, \\
0, & n-k+1\leq i \leq n,
\end{cases}
\]
where $\psi(i)\in\mathbb{C}$.
\end{proposition}
\begin{proof}The proof is carrying out by applying similar arguments as in the proof of the  previous theorem.
\end{proof}

Now we consider a homogeneous Nijenhuis operator with degree $1$.

\begin{theorem}  Let $P_1$ be a homogeneous Nijenhuis operator with degree $1$ on the null-filiform associative algebra. Then

\begin{enumerate}
\item[(a)] If $\psi(1)\neq0$ then $P_0(e_i)=\psi(i)e_{i+1}$, where the function $\psi$ satisfies the next recurrence formula
\[\psi(i)=\frac{\psi(1)\psi(i-2)}{\psi(1)+\psi(i-2)-\psi(i-1)}, \quad 3\leq i\leq n-1;\]
\item[(b)] If $\psi(1)=0$ and $\psi(t)\neq0$ for some $2\leq t\leq \lfloor (n-2)/2\rfloor$ then
\[P_0(e_i)=
\begin{cases}
\frac{(t+1)\psi(t)}{i+1}e_{i+1}, & i+1=0 \ (\bmod \ t+1), \\
0, & \text{otherwise};
\end{cases}
\]
\item[(c)] If $\psi(t)=0$ for all $1\leq t\leq \lfloor (n-2)/2\rfloor$ then

\[P_0(e_i)=
\begin{cases}
\psi(r)e_{r+1}, & i=r, \\
0, & \text{otherwise},
\end{cases}
\]
for some fixed $\lfloor (n-2)/2\rfloor < r \leq n-1$.
\end{enumerate}

\end{theorem}

\begin{proof}The proof is carrying out by applying similar arguments as in the proof of Theorem~\ref{d1}.
\end{proof}

\subsection{Average operator}
 \begin{theorem}  Let $P_0$ be a homogeneous Average operator with degree $0$ on the  null-filiform associative algebra. Then

\begin{enumerate}
\item[(a)] If $\psi(t)\neq0$ for some $1\leq t\leq \lfloor n/2\rfloor$ then
\[P_0(e_i)=
\begin{cases}
\psi(t)e_i, & i=0 \ (\bmod \ t), \\
0, & \text{otherwise};
\end{cases}
\]
\item[(b)] If $\psi(t)=0$ for all $1\leq t\leq \lfloor n/2\rfloor$ then

\[P_0(e_i)=
\begin{cases}
\psi(r)e_r, & i=r, \\
0, & \text{otherwise},
\end{cases}
\]
for some fixed $\lfloor n/2\rfloor < r \leq n$.
\end{enumerate}

\end{theorem}

\begin{proof}

 Consider  \[P(e_i)P(e_j)=\psi(i)\psi(j)e_{i+j}. \quad 2\leq i+j\leq n.\]
On the other hand
\[P(e_iP(e_j))=P(\psi(j)e_{i+j})=\psi(j)\psi(i+j)e_{i+j}.\]

The comparison of both linear combinations implies that:
\begin{equation}\label{i}
  \psi(i)\psi(j)=\psi(j)\psi(i+j), \quad 2\leq i+j\leq n.
\end{equation}

\begin{enumerate}
  \item[(a)] Let $\psi(1)\neq0$. Then we can obtain that
$\psi(i)=\psi(1), \quad 2\leq i\leq n$.
 Let us suppose that $\psi(1)=0$ and
$\psi(t)\neq0$ for some $2\leq t\leq \lfloor n/2\rfloor$. From the system of
Eq.~\eqref{i} it is easy to check that the following equalities hold:
\[\psi(it)=\psi(t), \qquad 2\leq i\leq [n/t].\]

If $\{s/t\}\neq0$ we always can write $s=it+q$, where
$t+1\leq s\leq n,\ 1\leq q\leq t-1$. From the system of Eq.~\eqref{i} we
have
\[\psi(s)=\psi(it+q)=\frac{\psi(q)\psi(it)}{\psi(it)}=\psi(q)=0.\]

So we have \[\psi(s)=0, \quad k+1\leq s\leq n, \quad  \{s/t\}\neq0.\]

  \item[(b)] Let $\psi(r)=0, \ 1\leq r\leq t-1$ and
$\psi(t)\neq0, \lfloor n/2\rfloor<t\leq n$. As in the  previous case we can
always write $s=t+q$ where $1\leq q\leq n-t$ then
\[\psi(s)=\psi(t+q)=\frac{\psi(t)\psi(q)}{\psi(t)}=\psi(q)=0.\] Hence
we have $\psi(s)=0$ where $t+1\leq s\leq n$.

\end{enumerate}

 \end{proof}

\begin{proposition}Let $P_k$ be a homogeneous Average operator with degree $k \ (k\geq\lfloor n/2\rfloor)$ on the null-filiform associative algebra.
Then
\[P_k(e_i)=
\begin{cases}
\psi(i)e_{i+k}, & 1\leq i \leq n-k, \\
0, & n-k+1\leq i \leq n,
\end{cases}
\]
where $\psi(i)\in\mathbb{C}$.

\end{proposition}
\begin{proof}The proof is carrying out by applying similar arguments as in the proof of the previous theorem.
\end{proof}

Finally we consider homogeneous Average operator with degree $1$.

\begin{theorem}  Let $P_1$ be a Average operator  with degree $1$ on the null-filiform associative algebra. Then
\begin{enumerate}
\item[(a)] If $\psi(1)\psi(2)\neq0$  then $P_1(e_i)=\psi(1)e_{i+1}$;
\item[(b)] If $\psi(1)\neq0$ and $\psi(2)=0$ then
\[P_1(e_i)=
\begin{cases}
\psi(1)e_{i+1}, & i  \ \text{odd}, \\
0, & \text{otherwise};
\end{cases}
\]
\item[(c)] If $\psi(1)=0$ and $\psi(t)\neq0$ for some $2\leq t\leq \lfloor (n-2)/2\rfloor$ then
\[P_1(e_i)=
\begin{cases}
\psi(t)e_{i+1}, & i+1=0 \ (\bmod \ t+1), \\
0, & \text{otherwise};
\end{cases}
\]
\item[(d)] If $\psi(t)=0$ for all $1\leq t\leq \lfloor (n-2)/2\rfloor$    then
\[P_1(e_i)=
\begin{cases}
\psi(r)e_{r+1}, & i=r, \\
\psi(r+1)e_{r+2}, & i=r+1, \\
0, & \text{otherwise},
\end{cases}
\]
for some fixed $\lfloor (n-2)/2\rfloor < r \leq n-1$.
\end{enumerate}

\end{theorem}
\begin{proof}The proof is carrying out by applying similar arguments as in the proof of  Theorem~\ref{d1}.
\end{proof}


\end{document}